\documentclass[10pt,a4paper]{article}
\usepackage[utf8]{inputenc}
\usepackage{amsmath}
\usepackage{amsfonts}
\usepackage{amssymb,amsthm}
\usepackage{graphicx,epstopdf}

\usepackage[font={small}]{caption}

\newtheorem{theorem}{Theorem}
\newtheorem{proposition}{Proposition}
\newtheorem{remark}{Remark}
\newcommand{\xkd}[1]{x_{#1}^\delta} 
\newcommand{\xk}[1]{x_{#1}} 
\newcommand{\zkd}[1]{z_{#1}^\delta} 
\newcommand{\yd}{y^\delta} 
\newcommand{\xd}{x^\dagger} 

\title{Optimal-order convergence of Nesterov acceleration 
for linear ill-posed problems}
\author{Stefan Kindermann\thanks{Industrial Mathematics Institute, Johannes Kepler University
 Linz, Linz, Austria ({\tt kindermann@indmath.uni-linz.ac.at}). }} 
\date{}

\begin{document}
\maketitle
\begin{quote}{\em 
This article is dedicated to A.~Neubauer on the occasion of his 60th birthday. 
His analytical insight shall be our unmatched  benchmark.}
\end{quote} 

\begin{abstract}
We show that Nesterov acceleration 
 is an optimal-order 
iterative regularization method for linear ill-posed problems 
provided that a parameter is chosen accordingly to the smoothness 
of the solution. This result is proven both for an a priori stopping rule and 
for the discrepancy principle. 
The essential tool to obtain this result is a
representation of the residual polynomials via Gegenbauer polynomials. 
\end{abstract} 

\section{Introduction}
One option to calculate a regularized solution to 
 an ill-posed operator equation $A x = \yd$, with $A: X\to Y$ and 
 $X,Y$ being Hilbert spaces, is to employ iterative regularization 
 schemes, where approximate solutions $\xkd{k}$ are calculated iteratively
 combined with a stopping rule as regularization parameter choice.  
 The simplest one being  Landweber iteration (cf., e.g.,~\cite{EHN96}), which has 
 the downside of being rather slow. To speed up convergence, acceleration 
 schemes may be used such as the following {\em Nesterov acceleration}:  
\begin{align} \label{Nfirst} 
\begin{split} 
\xkd{k+1} &= \zkd{k} + A^*(\yd - A \zkd{k}), \qquad k \geq 1  \\
\zkd{k} & = \xkd{k} + \alpha_k (\xkd{k}  - \xkd{k-1}), \qquad x_0 = 0, x_1 = A^*\yd,
\end{split} 
\end{align} 
where $\|A^*A\| \leq 1$ is assumed and where 
the sequence $\alpha_k$ is chosen, for instance, as 
\begin{align}\label{achoice}
 \alpha_k = \frac{k-1}{k+\beta},  \qquad k \geq 1, \quad  \beta > -1.  \end{align}
Here $\beta$ is a parameter; common choices are,    for example,   $\beta = 1$ or $\beta=2$.  
 We remark that other 
alternatives for the sequence
   $\alpha_k$ are possible as well, but for the main analysis  of this paper we only consider  \eqref{achoice}.

This iteration (in a general nonlinear context) was suggested by Yurii Nesterov 
for general convex optimization problems \cite{Ne83}. 
It is an instance of a method that achieves the 
best rate of convergence (in the sense of 
objective function decrease) that is generally possible for a  first-order 
method. The Nesterov acceleration 
can be employed to speed up 
convergence of gradient methods
in nonlinear or convex optimization. 
A particular successful instance is the FISTA algorithm of Beck and Teboulle
\cite{BeTe09} for nondifferentiable convex optimization. 

In the realm of ill-posed problems, Hubmer and Ramlau 
\cite{HuRa18} performed a 
convergence  analysis for nonlinear problems, and shows 
the efficiency of the method.

The background and main motivation of the 
present article is the recent interesting analysis of 
Neubauer \cite{Ne17} for ill-posed problems 
in the linear case. He showed that \eqref{Nfirst}
is an iterative regularization scheme and, 
more important proved convergence rates, 
which are  of optimal order only  for a priori parameter choices and 
in case of low smoothness of the solution and suboptimal else.
What is puzzling is that the method shows 
a quite unusual 
 "semi-saturation" phenomenon
 (we explain this term below in Section~\ref{sec:semisat}).
 
% Our contribution 
 
Our contribution in this article is 
twofold: At first, we prove a formula 
for the residuals of the iteration \eqref{Nfirst} 
involving Gegenbauer polynomials. 
On this basis, we can build a convergence 
rate analysis, which improves and extends 
the results of Neubauer. In particular, 
we show that the method can always 
be made an optimal-order method if 
the parameter $\beta$ is chosen accordingly 
to the (H\"older-)smoothness index of the solution. This result holds 
 for both an a priori stopping rule and for the discrepancy principle. 

Our analysis also explains the quite 
nebulous role that this parameter plays in the 
iteration, as it is related to the index of 
the orthogonal polynomials appearing in the 
residual formula. 

Moreover, the above mentioned residual representation 
also clearly elucidates the semi-saturation phenomenon 
because the iteration can be interpreted as a mixture
of a saturating  iteration (Brakhage's $\nu$-method) and a non-saturating one (Landweber method).

In the following we employ some standard notation of regularization theory as in \cite{EHN96}: 
$\delta = \|A \xd -\yd\|$ is the 
noise level and $\xd$ denotes the minimum-norm solution to the 
operator equation $A x = y$ with exact data $y= A\xd$. 
The index $\delta$ of $\yd$ indicates noisy data, and analogous,  
$\xkd{k}$ denotes  the iterates of
\eqref{Nfirst} with noisy data $\yd$, 
while the lack of $\delta$  indicates exact data $y$ and correspondingly 
the iteration $\xk{k}$  with exact data $y$ in place of $\yd$ in \eqref{Nfirst}.

\section{Residual polynomials for Nesterov acceleration} 
Our work follows the general theory of spectral filter-based regularization methods
as in \cite{EHN96}, where the convergence 
analysis results from  estimates of the corresponding 
filter function. The first main result, 
Theorem~\ref{thone} is quite useful for this 
purpose as it represents the residual function 
in terms of known polynomials.

The iteration \eqref{Nfirst} is a Krylov-space method, and the residual can 
be expressed as 
\[ \yd- A \xkd{k} =: r_k(AA^*)\yd \]
with the residual polynomials satisfying the recurrence relation (cf.~\cite{Ne17})
\begin{align}\label{resi}
\begin{split} 
r_{k+1}(\lambda) &= (1-\lambda) \left[ r_k(\lambda) + \alpha_k(r_{k}(\lambda) - r_{k-1}(\lambda)) \right],
\quad k\geq 1, \\
\qquad  r_0(\lambda) &= 1, \qquad r_1(\lambda) = (1-\lambda). 
\end{split} 
\end{align}
This is a simple consequence of the definition in  \eqref{Nfirst}. 
The $k$-th iterate  can be expressed via spectral filter functions  
\[ \xkd{k} = g_k(A^*A) A^*\yd, \qquad g_k(\lambda) := \frac{1- r_k(\lambda)}{\lambda}. \]
%As mentioned before, the convergence analysis is entirely based  on estimates of the functions $r_k$ and $g_k$.  

Observe that the three-term recursion  \eqref{resi} is {\em not} of the form 
to apply Favard's theorem \cite{Fa35}, hence 
$r_k$ is  {\em not} any orthogonal polynomial with respect to some weight functions. 
(Note that Favard's theorem  fully characterizes  three-term  recurrence relations that 
lead to orthogonal polynomials).    

Before we proceed, we may compare the residual polynomials with other well-known cases. 
For classical Landweber iteration \cite{EHN96}, which is obtained by setting $\alpha_k = 0$ 
and thus $\zkd{k} = \xkd{k}$, the corresponding residual functions $r_k = :r_k^{(LW)}$ is 
\[ r_{k}^{(LW)}(\lambda) = (1-\lambda)^k. \] 
On the other hand, another class of  well-known iteration methods for ill-posed problems  that are  based  
on orthogonal polynomials are two-step  semiiterative methods \cite{Ha91}.
They have the form 
\begin{align*} 
\xkd{k+1} &= \xkd{k} + \mu_{k+1}(x_{k} - x_{k-1}) + \omega_{k+1} A^*(\yd - A x_{k}), \qquad k >1, 
\end{align*} 
where $\mu_k$ and $\omega_k$ are appropriately chosen sequences. 
The corresponding residual functions satisfy the recurrence relation 
\begin{align} 
r_{k+1}(\lambda) &= (1 - \omega_{k+1} \lambda) r_{k}(\lambda)+ 
\mu_{k+1}(r_{k} - r_{k-1}), \qquad k >1,  \label{recursion} 
\end{align} 
and thus, $r_k(\lambda)$ form a sequence of orthogonal polynomials.
Of special interest in ill-posed problems are the $\nu$-methods of Brakhage~\cite{Br87,Ha91}, defined by 
the sequences, for $k>1$,
\begin{align*} 
 \mu_{k+1} &= \frac{(k-1) (2k -2) (2k + 2 \nu -1)}{(k+2 \nu-1) (2k + 4 \nu-1) (2k + 2 \nu -3) }, \\ 
\qquad 
\omega_{k+1} &= 4 \frac{(2k + 2 \nu-1)(k + \nu -1)}{(k + 2 \nu-1) (2k + 4 \nu -1) }, 
\end{align*}
the initial values $x_0 = 0,$ $x_1 = \frac{4 \nu +2}{4 \nu +1} T^* \yd$, and 
with $\nu>0$ a user-selected parameter.  
The associated residual polynomials  $r_k = :r_k^{(\nu)}$ related to 
 \eqref{recursion}  
with $r_0 = 1$,  $r_1 = 1- \lambda \frac{4 \nu +2}{4 \nu +1},$
have the representation \cite{Br87}
\[ r_k^{(\nu)}(\lambda) = \frac{C_{2k}^{(2\nu)}(\sqrt{1-\lambda})}{C_{2k}^{(2\nu)}(1)}, \]
where $C_n^{(\alpha)}$ denotes the Gegenbauer polynomials (aka. ultraspherical polynomials); cf.~\cite{AbSt64}. 

We now obtain the corresponding representation for the Nesterov residual polynomials, 
which is the basis of this article.  
\begin{theorem}\label{thone}
Let $\beta >-1$.
The residual polynomials for the Nesterov acceleration \eqref{Nfirst} with \eqref{achoice} are 
\begin{equation}\label{resibesi} 
 r_k(\lambda)  = (1-\lambda)^\frac{k+1}{2} 
  \frac{C_{k-1}^{(\frac{\beta +1}{2})}(\sqrt{1-\lambda})}{C_{k-1}^{(\frac{\beta +1}{2})}(1)}, \qquad k \geq 1, 
\end{equation} 
with the Gegenbauer polynomials $C_n^{(\alpha)}$. 
\end{theorem} 
 
\begin{proof}
Defining $h_k(\lambda) = r_k(\lambda)(1-\lambda)^{-\frac{k+1}{2}}$ and multiplying \eqref{resi} 
by $(1-\lambda)^{-\frac{k+2}{2}}$ leads to the relation 
\begin{align} 
h_{k+1}(\lambda) &= (1+\alpha_k) \sqrt{1-\lambda} h_k(\lambda)
 - \alpha_k h_{k-1}(\lambda), \qquad k \geq 2,
\label{hrec}  \\ 
h_1(\lambda) &= 1, \qquad  h_2(\lambda)  = \sqrt{1-\lambda}. \nonumber 
\end{align} 
We note that $C_n^{(\frac{\beta +1}{2})}(x)$ satisfy the recursion relation 
(cf.~\cite[p.~782]{AbSt64})
\begin{equation}\label{recgeg}
\begin{split} 
 C_k^{(\frac{\beta +1}{2})}(x) &= x c_k C_{k-1}^{(\frac{\beta +1}{2})}(x)
- d_k C_{k-2}^{(\frac{\beta +1}{2})}(x), \qquad k \geq 2 \\
 C_0^{(\frac{\beta +1}{2})}(x) &= 1, \qquad  C_1^{(\frac{\beta +1}{2})}(x) = (\beta +1) x 
\end{split} 
\end{equation} 
with 
\[ c_k = \frac{2k + \beta-1}{k}, \qquad d_k =  \frac{k + \beta-1}{k}. \]
Using the recurrence relation with $x = 1$, leads to 
\begin{align*} 
 C_k^{(\frac{\beta +1}{2})}(1) &= 
c_k  C_{k-1}^{(\frac{\beta +1}{2})}(1)\left( 
1 - \theta_k^{-1}  \right)= 
d_k C_{k-2}^{(\frac{\beta +1}{2})}(1)\left( \theta_k -1 \right)
\end{align*} 
with 
\[ \theta_k := \frac{c_k  C_{k-1}^{(\frac{\beta +1}{2})}(1)}{d_k  C_{k-2}^{(\frac{\beta +1}{2})}(1) }. \]
Dividing \eqref{recgeg} by  $C_k^{(\frac{\beta +1}{2})}(1)$ and using this relation yields
\begin{align}  \frac{ C_k^{(\frac{\beta +1}{2})}(x)}{ C_k^{(\frac{\beta +1}{2})}(1)}
&= 
\frac{ C_{k-1}^{(\frac{\beta +1}{2})}(x)}{ C_{k-1}^{(\frac{\beta +1}{2})}(1)}
\frac{1}{1  - \theta_k^{-1}} - 
\frac{ C_{k-2}^{(\frac{\beta +1}{2})}(x)}{ C_{k-2}^{(\frac{\beta +1}{2})}(1)}
\frac{1}{\theta_k-1}.   \label{scaledrec} 
\end{align} 
By induction (or by well-known formulae \cite{AbSt64,Sz75}), it can easily  be verified that 
$ \frac{C_{k-1}^{(\frac{\beta +1}{2})}(1)}{ C_{k-2}^{(\frac{\beta +1}{2})}(1)} 
= \frac{k +\beta -1}{k-1},$
from which it follows that $\theta_k -1 = \alpha_k^{-1}$ as well as 
$1-\theta_k^{-1} = (1+\alpha_k)^{-1}$. Thus,  
$\frac{ C_k^{(\frac{\beta +1}{2})}(x)}{ C_k^{(\frac{\beta +1}{2})}(1)}$ satisfies the 
same recursion as $h_{k+1}(\lambda)$, 
  and the corresponding initial values for $k = 0,1$ agree when setting $x = \sqrt{1-\lambda}$.  
This allows us to conclude that 
\[ h_k(\lambda)  =  \frac{ C_{k-1}^{(\frac{\beta +1}{2})}(\sqrt{1-\lambda})}{ C_{k-1}^{(\frac{\beta +1}{2})}(1)},\]
which  proves the theorem. 
\end{proof} 

This theorem relates the residual function of Nesterov acceleration to other known iterations. In particular, the residual $r_k$ is roughly the product of
that of $\frac{k}{2}$ Landweber iterations and that of $\frac{k}{2}$  iterations of a $\nu$-method with 
$\nu = \frac{\beta +1}{4}$. 

\begin{remark} 
Gegenbauer polynomials are special cases of Jacobi polynomials and themselves 
embrace several other orthogonal polynomials
as special cases. Certain values of $\beta$ in \eqref{Nfirst} lead to various 
specializations in \eqref{resibesi}: The choice  $\beta = 0$ leads to Legendre polynomials, 
the often encountered choice $\beta = 1$ leads 
to Chebyshev polynomials of the second kind~\cite{AbSt64}. 

We note that the result of Theorem~\ref{thone} even holds for  $\beta = -1$. 
In this case, only $\alpha_1$ is not well-defined, but it is always 
$0$ for $\beta >-1$. 
Thus, we may extend the definition of the iteration to $\beta = -1$ 
by setting $\xkd{k}:= \lim_{\beta \to -1} \xkd{k}$. 
(This just amounts to slightly modifying \eqref{achoice} by 
 setting $\alpha_1 = 0$ for $k=1$; the remaining iteration is well-defined by \eqref{Nfirst} and \eqref{achoice}.) 
 In this case, we may use \cite[Eq.~22.5.28]{AbSt64} to conclude that 
 the resulting polynomials are Chebyshev polynomials of the first kind. 
\end{remark}

Before we proceed with the convergence analysis, we state for generality the 
corresponding theorem for Nesterov iteration with a general sequence $\alpha_k$. 

\begin{theorem}
Consider the iteration \eqref{Nfirst} with a positive sequence $\alpha_k$. 
Then the corresponding residual function can be expressed as 
\begin{equation} \label{genrep}
 r_k(\lambda) = (1- \lambda)^\frac{k}{2} \frac{P_{k}(\sqrt{1-\lambda})}{P_{k}(1)},
  \qquad k \geq 1, 
 \end{equation} 
where $P_k$ is a sequence of orthogonal polynomials obeying the recurrence relation 
\begin{align}\label{genrec}  
\begin{split}
 P_{k+1}(x) &= c_k x P_{k}(x) - d_k P_{k-1}(x), \qquad  k \geq 1 \\ 
 P_0(x) & = 1, \qquad P_1(x) = c_0 x  
 \end{split}
 \end{align} 
with $c_n$ and $d_n$ recursively defined to satisfy 
\begin{align}\label{cda}
\begin{split}
 \frac{c_1 c_0}{d_1} & = 1 + \frac{1}{\alpha_1} \\  
 \frac{c_k c_{k-1}}{d_k} &= (1 + \frac{1}{\alpha_k})(\alpha_{k-1} +1) \qquad k \geq 2. 
 \end{split} 
  \end{align} 
Conversely, given a sequence of orthogonal polynomials defined by the 
recurrence relation \eqref{genrec} with given sequences $c_n,d_n$. 
Then there exists a sequence $\alpha_k$ (defined via \eqref{cda}) such that the corresponding 
Nesterov iteration \eqref{Nfirst} has a residual function as in \eqref{genrep}.
\end{theorem}
\begin{proof} 
The function $h_k(\lambda):= r_k(\lambda) (1-\lambda)^\frac{k}{2}$ satisfies 
the recursion \eqref{hrec} with $h_0(\lambda) = 1$ and $h_1(\lambda) = \sqrt{1-\lambda}$
and for $k\geq 1$.  
As in the proof of Theorem~\ref{thone}, we may conclude that \eqref{genrec} 
leads to a similar  recursion as \eqref{scaledrec}: 
\[ \frac{P_{k+1}(x)}{P_{k+1}(1)} = 
\frac{P_{k}(x)}{P_{k}(1)} \frac{1}{1 -\theta_k^{-1}} -
\frac{P_{k-1}(x)}{P_{k-1}(1)}\frac{1}{\theta_k-1}, 
\qquad k \geq 1. \] 
with 
\[ \theta_k = \frac{c_k  P_{k}(1)}{d_k  P_{k-1}(1) }, \qquad k \geq 1. \]
From  \eqref{genrec} we can conclude by some algebraic manipulations that 
\[ \theta_{k} = \frac{c_k c_{k-1}}{d_k}\left( 1- \theta_{k-1}^{-1} \right), \qquad k\geq 2.  \] 
If \eqref{cda} holds, then 
from the recursion 
for $\theta_k$, it  follows that we can perform an induction step 
following that 
 $\frac{1}{\theta_{k-1} -1} = \alpha_{k-1}$ implies 
$\frac{1}{\theta_{k} -1} = \alpha_{k}$. Since  
 $\frac{1}{\theta_{1} -1} = \alpha_{1}$ by definition, we obtain that 
 $h_k(\lambda)$ and $\frac{P_{k}(x)}{P_{k}(1)}$ satisfy identical recursions 
 and have identical initial conditions with the  setting $x = \sqrt{1-\lambda}$.
 
Conversely, if  \eqref{genrec} is given and the sequence $\alpha_k$ is recursively  defined 
by~\eqref{cda}, then it follows in a similar manner  that $\frac{P_{k}(x)}{P_{k}(1)}$ has the same 
recursion and initial conditions as $h_k(\lambda)$ and thus both functions agree. 
\end{proof} 
The polynomials $P_k(x)$ in this theorem correspond to 
$x C_{k-1}^{\frac{\beta+1}{2}}(x)$ in Theorem~\ref{thone}.

As an illustration, we may consider the peculiar choice of $\alpha_k$ 
in Nesterov's original paper~\cite{Ne83}, which is also used in the well-known 
FISTA iteration \cite{BeTe09}: First, a sequence is defined recursively, 
\[ t_{k+1} = \frac{1}{2}\left( 1 + \sqrt{1 + 4 t_k^2} \right), \qquad t_1 = 1,\]
and then the sequence $\alpha_k$ is given by 
\[ \alpha_k = \frac{t_k-1}{t_{k+1}}. \]
Note that  $t_{k+1}$ is the positive root of the equation
$ t_{k+1}(t_{k+1}-1) = t_{k}^2.$  Using this identity, we may calculate that 
\[ 
(1 + \frac{1}{\alpha_k})(\alpha_{k-1} +1) = 
\frac{t_k}{t_{k-1}} (1 + \frac{t_k}{t_{k+1}}) ( 
1 +\frac{t_{k-1}}{t_{k}}). \] 
Thus, coefficients for a recurrence formula for orthogonal polynomials 
 that correspond to such an 
iteration are 
\[ c_k = 1 + \frac{t_k}{t_{k+1}},  \qquad d_k = c_{k-1} -1. \] 
However, this does not seem to be related to any common polynomial family,
 to the knowledge of the 
author. 

On the other hand, we may design Nesterov iterations from the recurrence 
relation of classical polynomials. For instance, the Hermite polynomials 
obey a relation \eqref{genrec}   with $c_k = 2,$ $d_k = 2k$. 
Thus, the sequence $\alpha_k$ has to satisfy the recursion 
\[  a_{k}:= \frac{1 + a_{k-1}}{\frac{2}{k} - a_{k-1} - 1}.  \] 
We do not know if this is of any use, though.

\section{Convergence analysis}
We consider the iteration \eqref{Nfirst} with the usual $\alpha_k$-sequence \eqref{achoice} and show that 
it is an optimal-order regularization methods (of course, when combined with a stopping rule). 
%The analysis follows the general theory of \eqref{EHN96}. 

\subsection{Convergence rates and semi-saturation}\label{sec:semisat}
In the classical analysis of regularization schemes \cite{EHN96}, one tries 
to bound the error in terms of the noiselevel $\delta$: $\|\xkd{k(\delta)} -\xd\| \leq f(\delta)$, 
where $f$ is some function decreasing to $0$ with $\delta \to 0$.
Often,  H\"older-type 
rates are considered with $f(\delta) = C \delta^\xi$. For such estimates, one has to 
impose smoothness conditions in form of a source condition 
\begin{align}\label{source}  
\xd = (A^*A)^\mu \omega,   \qquad \|\omega\|< \infty,\quad  \mu >0.  
\end{align} 
It is also well-known \cite{EHN96} that the optimal rate of convergence under \eqref{source} 
is of the form 
\[ \|\xkd{k(\delta)} -\xd\|  \leq O(\delta^{\frac{2\mu}{2\mu +1}}), \]
and a regularization scheme that achieves this bound is called of optimal order. 

The phenomenon of saturation is the effect that for certain regularization method, 
the convergence rate $f(\delta)$ does not improve even when the smoothness is 
higher, i.e., $\mu$ is larger. This happens, for instance for 
Tikhonov regularization at $\mu = 1$ or for the $\nu$-methods at $\mu = \nu$; see~\cite{EHN96}. 
 
For the Nesterov iteration \eqref{Nfirst}, a detailed analysis  
has been performed by Neubauer \cite{Ne17} with the result 
that, assuming  a usual source condition \eqref{source} 
and an appropriate a priori stopping rule, the resulting iterative regularization 
scheme is of optimal order for $\mu \leq \frac{1}{2}$, and, 
for $\mu > \frac{1}{2}$, the convergence rates improve with $\mu$ but in 
a suboptimal way. 
More precisely, the convergence rates proven in \cite{Ne17} are 
\[ \|\xkd{k(\delta)} -\xd \| =  \begin{cases} O(\delta^\frac{2\mu}{2\mu+1}) & \mu \leq \frac{1}{2}, 
\\ 
O(\delta^\frac{2\mu+1}{2\mu+3}) & \mu > \frac{1}{2}.  \end{cases}
\]  
Thus, contrary to saturating methods, the order still improves beyond the "saturation 
index" $\mu = \frac{1}{2}$ but in a suboptimal way. This is what we call "semi-saturation", 
and, to the knowledge of the author, 
 this has not been observed yet for a classical regularization method.   
A further result of \cite{Ne17} is that  using  the discrepancy principle as 
stopping rule, convergence rates are proven, which are, however, always suboptimal.

Our second main contribution is  an improvement of Neubauer's result in the 
sense that we show that the Nesterov iteration is of optimal order 
for a smoothness index $\mu \leq \frac{\beta +1}{4}$ with an a priori stopping rule.
 Moreover, contrary to 
\cite{Ne17}, we also obtain optimal-order rates with the discrepancy principle 
provided  that $\mu \leq \frac{\beta -1}{4}$. 
These findings allows one to achieve always optimal-order 
convergence provided $\beta$ is chosen sufficiently large.

Moreover, the phenomenon of semi-saturation is made transparent by 
referring to the representation in Theorem~\ref{thone}: 
The residual is a product of Landweber-type and $\nu$-type residuals, 
and keeping in mind that Landweber iteration does not show 
saturation 
for H\"older indices while the $\nu$-method do, it is clear that a 
product as in \eqref{resibesi} leads to the above described semi-saturation.

\subsection{Convergence analysis} 
In this section we perform a convergence analysis for the iteration~\eqref{Nfirst}. 
By Theorem~\ref{thone}, we may base our investigation on the known 
results for Landweber iteration and the $\nu$-methods.

We collect some useful known estimates:
\begin{equation}\label{Gegbound}
\left|\frac{C_{k-1}^{(\frac{\beta +1}{2})}(\sqrt{1-\lambda})}{C_{k-1}^{(
\frac{\beta +1}{2})}(1)} \right| \leq 1,  \qquad 0 \leq \lambda \leq 1, \beta >-1.  
\end{equation} 
This is well-known and follows from \cite[Eq. (7.33.1), (4.73)]{Sz75}.
From this we immediately obtain that 
\begin{equation}\label{resbound}  |r_k(\lambda)| \leq 1, \qquad  0 \leq \lambda <1, \quad \beta >-1, 
 \end{equation}
which has  already been shown in \cite{Ne17}. 
Moreover, we may conclude from \eqref{Gegbound} and \eqref{resibesi}
as well that 
\begin{equation}\label{rescon} 
 \lim_{k\to \infty} r_k(\lambda) \to 0, \qquad  0 < \lambda <1.
 \end{equation}

%
%\begin{align*} 
%(1-\lambda)^\frac{k+1}{2} &\leq 1 \qquad  0 \leq \lambda \leq 1 \\ 
%\frac{C_{k-1}^{(\frac{\beta +1}{2})}(\sqrt{1-\lambda})}{C_{k-1}^{(
%\frac{\beta +1}{2})}(1)} \leq 1 & \qquad 0 \leq \lambda \leq 1  \\ 
%(1-\lambda)^{k+1}{2} \lambda^\mu &\leq \begin{cases} ( \frac{k+1}{2} +1)^{-\mu} & 0 <\mu \leq 1 \\ 
%                        \mu^\mu (\frac{k+1}{2} +1)^{-\mu} &  \mu >1 \end{cases} \\
%\frac{C_{k-1}^{(\frac{\beta +1}{2})}(\sqrt{1-\lambda})}{C_{k-1}^{(
%\frac{\beta +1}{2})}(1)}  \lambda^\mu &\leq 
%\begin{cases} c_\mu (k-1)^{2\mu}      & \mu \leq   \frac{\beta +1}{2} \\
% c_\mu (k-1)^{2 \frac{\beta +1}{2}} & \mu >   \frac{\beta +1}{2} 
% \end{cases} 
%\end{align*} 

Recall that  we denote by $\xk{k}$ the iteration 
with $\yd$ replaced by the exact data. As usual, this allows one to split 
the total error into an approximation and stability term. 
We estimate the stability term: 
\begin{proposition}
Let  $\|A^*A\| \leq 1$  and define  $\xkd{k}$ by  
\eqref{Nfirst} \eqref{achoice} with $\beta >-1$. 
Let  
$\xk{k}$ be the corresponding noise-free iteration with $\yd$ replaced by $y = A \xd$. 
Then we have the estimate 
\begin{equation}\label{estdata}
\| \xkd{k} - \xk{k}\| \leq \sqrt{2} \sqrt{(k-1)^2 + \frac{k+1}{2}} \delta \leq  
\sqrt{2} k \delta. 
\end{equation} 
\end{proposition}
\begin{proof}
Following \cite{EHN96}, it is enough to estimate 
\[ g_k(\lambda) = \frac{1 -r_k(\lambda)}{\lambda} = r_k'(\tilde{\lambda}),\]
where we used the mean value theorem with $\tilde{\lambda} \in (0,\lambda)$. 
The derivative  may be calculated from \eqref{resibesi} as 
\begin{align*}
r_k'(\lambda) &= \frac{k+1}{2} (1-\lambda)^{\frac{k-1}{2}} 
  \frac{C_{k-1}^{(\frac{\beta +1}{2})}(\sqrt{1-\lambda})}{C_{k-1}^{(
\frac{\beta +1}{2})}(1)}  - \frac{1}{2} (1-\lambda)^{\frac{k}{2}} 
  \left(\frac{[C_{k-1}^{(\frac{\beta +1}{2})}](\sqrt{1-\lambda})}{C_{k-1}^{(
\frac{\beta +1}{2})}(1)}\right)'. 
\end{align*} 
We use Markov's inequality (cf.~\cite[Eq.~(6.16)]{EHN96}) and \eqref{Gegbound} to conclude that 
\[  \left|\left(\frac{[C_{k-1}^{(\frac{\beta +1}{2})}](\sqrt{1-\lambda})}{C_{k-1}^{(
\frac{\beta +1}{2})}(1)}\right)'|\right| \leq 
2 (k-1)^2 \max_{0\leq \lambda \leq 1}
\left| \frac{C_{k-1}^{(\frac{\beta +1}{2})}(\sqrt{1-\lambda})}{C_{k-1}^{(\frac{\beta +1}{2})}(1)}
\right| \leq 
2 (k-1)^2. 
\]
Thus,
\[ |g_k(\lambda)|  \leq \frac{k+1}{2}  + (k-1)^2. \]
The result now follows with \cite[Theorem~4.2]{EHN96} and \eqref{resbound}. 
\end{proof} 
Note that this estimate is a slight improvement compared to the corresponding estimate in 
\cite[Equation (3.2)]{Ne17}, which has $2 k \delta$ on the right-hand side, 
similar as for the $\nu$-methods.

From this we may conclude convergence: 
\begin{theorem}\label{th:conv}
Let $\|A^*A\| \leq 1$ and $\beta>-1$. If the iteration is stopped at a 
stopping index $k(\delta)$ that satisfies $k(\delta) \delta \to 0$ 
and $k(\delta) \to \infty$ as $\delta \to 0$, then 
we obtain convergence 
\[ \xkd{k(\delta)} \to \xd. \]
\end{theorem}
\begin{proof} 
We estimate 
\begin{align*} 
 \|\xkd{k(\delta)} -\xd\| & \leq 
\|\xkd{k(\delta)} -\xk{k(\delta)} \| + \|\xk{k(\delta)} -\xd\| \leq 
\sqrt{2} k(\delta) \delta + r_{k(\delta)}(A^*A)\xd. 
\end{align*} 
The first term converges to $0$ by assumption on $k(\delta)$ and the 
second term does so because $k(\delta) \to \infty$ and by the dominated 
convergence theorem using \eqref{resbound}, \eqref{rescon} as in \cite{EHN96}.
\end{proof}

We now consider  convergence rates, and for this, 
the following rather deep estimate  for  orthogonal polynomials is needed.
It was derived by  Brakhage \cite{Br87} as well as by  Hanke \cite[Appendix~A.2]{EHN96}, \cite{Ha91} on basis of 
Hilb-type estimates for Jacobi polynomials.
\begin{proposition}
Let  $\beta >-1$. 
Then there is a constant $c_\beta$ with 
\begin{equation}\label{rateeq}
\left|\lambda^\frac{\beta +1}{4} \frac{C_{k}^{(\frac{\beta +1}{2})}(\sqrt{1-\lambda})}{C_{k}^{(\frac{\beta +1}{2})}(1)}\right|
\leq c_\beta k^{-2 \frac{\beta +1}{4}}, \qquad 0 \leq \lambda \leq 1. 
\end{equation} 
\end{proposition} 
\begin{proof}
For $k$ even, this is \cite[Eq.~(6.22)]{EHN96} (with $k$ there meaning $2k$ here), 
or \cite[Th.~4.1]{Ha91}. 
However, the result there is based on the Hilb-type formula
(\cite[Theorems~8.21.12, 8.21.13]{Sz75} which holds 
for all $k$ as in \cite[p.~170]{Br87}. Thus, by following the steps in \cite[Appendix~A.2]{EHN96},
the result is obtained. 
\end{proof} 
Note that in case $-1<\beta<1$, the constant $c_\beta$ may be explicitly calculated from 
\cite[Eq.~(7.33.5)]{Sz75}.

The corresponding estimates for the residuals of Landweber iteration
 are standard; cf.~\cite[Eq.~(6.8)]{EHN96}:
\begin{equation}\label{lwrates}
|\lambda^\mu   (1-\lambda)^k| \leq c_{\mu} (k+1)^{-\mu}. 
\end{equation}

As a consequence, we may state our main convergence rate result for an a priori stopping rule:
\begin{theorem}\label{th:six}
Let $\|A^*A\| \leq 1$ and $\beta>-1$, and suppose that a source condition 
\eqref{source} is satisfied with some $\mu >0$. 

\begin{enumerate}
\item If $\mu \leq \frac{\beta +1}{4}$ and the stopping index is chosen as 
\[ k(\delta) =  O(\delta^{-\frac{1}{2\mu+1}}), \]
then we obtain optimal order convergence 
\begin{align*} 
 \|\xkd{k(\delta)} -\xd\| \leq O(\delta^{\frac{2\mu}{2\mu+1}}). 
 \end{align*} 
\item  If $\mu > \frac{\beta +1}{4}$ and the stopping index is 
chosen as 
\begin{equation}\label{stopor}
 k(\delta) =  O(\delta^{-\frac{1}{\mu + \frac{\beta +1}{4}+1}}), 
 \end{equation} 
then we obtain suboptimal order convergence 
 \begin{align*} 
 \|\xkd{k(\delta)} -\xd\| \leq O(\delta^{\frac{\mu+\frac{\beta +1}{4}}{\mu+ \frac{\beta +1}{4}+1} }). 
 \end{align*}
\end{enumerate} 
\end{theorem} 
\begin{proof}
For $\lambda \leq 1$ the estimate  \eqref{rateeq} yields (by interpolation) 
and $(1-\lambda)^\frac{k+1}{2} \leq 1$ that 
\begin{equation}\label{anest}
 |r_k(\lambda)\lambda^\mu|  \leq C k^{-2\mu},  \qquad \mu \leq \frac{\beta+1}{4}. 
 \end{equation} 
In case of $ \mu > \frac{\beta+1}{4}$, we have with additionally using \eqref{lwrates}
\begin{align*} 
 |r_k(\lambda)\lambda^\mu| &\leq 
 |(1-\lambda)^\frac{k+1}{2} \lambda^{\mu-  \frac{\beta+1}{4}}| c_\beta k^{-2    \frac{\beta+1}{4}} \\
 & \leq 
 c_{\mu,\beta} (\frac{k+1}{2})^{-(\mu-  \frac{\beta+1}{4})} | c_\beta k^{-2    \frac{\beta+1}{4}} 
\leq C k^{-\left(\mu+  \frac{\beta+1}{4} \right) }.
\end{align*} 
The result now  follows by standard means: 
\begin{align*}
 \|\xkd{k} -\xd\| &\leq 
  \|\xkd{k} -\xkd{k}\| +  \|\xkd{k} -\xd\| \leq 
  \sqrt{2} k\delta + \|r_{k}(A*A)(A^*A)^\mu \omega\| \\
  & \leq 
\begin{cases} \sqrt{2} k\delta +   C k^{-2 \mu} & \mu \leq \frac{\beta+1}{4}, \\ 
k\delta +   C k^{-(\mu +   \frac{\beta+1}{4})} & \mu > \frac{\beta+1}{4}. 
\end{cases}
\end{align*}  
Solving for $k$ by equating the two terms in the last bounds yields the a priori parameter choice 
and the corresponding rates.   
\end{proof}  
 
These results correspond to those of Neubauer when $\beta = 1$.  However, for $\beta >1$ this is
an improvement as we obtain optimal-order convergence if $\beta$ is chosen larger than $4 \mu -1$. 
We note that in the optimal order case, the number of iteration needed is of order 
$O(\delta^{-\frac{1}{2\mu+1}})$, which is the same order as for semiiterative methods and
for the conjugate gradient method. Thus, the Nesterov acceleration 
certainly  qualifies being called a fast method. 
%We also note that the estimate for the number of iterations can slightly be improved by
%using the estimate $ \sqrt{2} \sqrt{(k-1)^2 + \frac{k+1}{2}} \delta = C_\beta k^{2\nu}$ 
%in \eqref{estdata}.  

\subsection{Discrepancy principle} 
With the improved estimates, we can as well strengthen the result of \cite{Ne17}  when
the iteration is combined with the well-known discrepancy principle. Recall that 
it defines a stopping index $k(\delta)$ a posteriori by the first (smallest) $k$ that fulfils the 
inequality 
\begin{equation}\label{discprin}
 \|A \xkd{k} - \yd\| \leq \tau \delta, \end{equation} 
where $\tau >1$ is fixed. The corresponding convergence rates can be obtained by a slight 
modification of the proof in \cite{Ne17} and the general theory in \cite{EHN96}. 

\begin{theorem}\label{th:seven}
Let $\|A^*A\| <1$, $\beta >-1$, and assume a source condition \eqref{source} satisfied.  If the iteration \eqref{Nfirst} is stopped by  the discrepancy principle \eqref{discprin}, then we 
obtain the following  convergence rates:
\begin{enumerate} 
\item If $\mu + \frac{1}{2} \leq \frac{\beta +1}{4}$, then we achieve 
optimal order convergence rates 
\begin{align*} 
 \|\xkd{k(\delta)} -\xd\| \leq O(\delta^{\frac{2\mu}{2\mu+1}}) 
 \end{align*} 
with stopping index being of the same order as in \eqref{stopor}. 
\item $\mu + \frac{1}{2} \geq \frac{\beta +1}{4}$, 
then we obtain that 
\[ k(\delta) = O(\delta^{-\frac{1}{\frac{1}{2}+ \mu + \frac{\beta+1}{4}}}) \]
and a rate of 
\begin{align*} 
 \|\xkd{k(\delta)} -\xd\| \leq O(\delta^{\frac{
 \mu + \frac{\beta+1}{4} -\frac{1}{2}  }{
 \mu + \frac{\beta+1}{4}+\frac{1}{2}}}).  
 \end{align*} 
\end{enumerate} 
\end{theorem}
\begin{proof} 
The proof  \cite[Theorem 4.1]{Ne17} only needs minor modifications. 
The estimate \cite[Eq.~(4.3)]{Ne17}
\[ \|\xk{k(\delta)} -\xd\| \leq \|r_{k(\delta)} (T^*T) w\|^\frac{1}{2\mu+1} 
\left((\tau +1) \delta\right)^\frac{2\mu}{2\mu+1} \]
is valid independent of our new rate results, hence it follows as in  \cite[Eq.~(4.4)]{Ne17}
that $\|\xk{k(\delta)} -\xd\| \leq o(\delta^\frac{2\mu}{2\mu+1} )$. It remains to 
estimate $\|\xkd{k(\delta)} -\xd\|$ by \eqref{estdata} combined with an 
upper bound for $k(\delta)$. Estimate \cite[Eq.~(4.2)]{Ne17} and the 
discrepancy principle yields 
\[ \tau \delta \leq \delta + \|(T^*T)^{\frac{1}{2}+ \mu} r_k(T^*T) w\| \] 
for $k = k(\delta)$. 
Using \eqref{anest} in case that $\mu +\frac{1}{2} \leq \frac{\beta+1}{4}$,
we obtain 
\[ (\tau-1) \delta \leq C k(\delta)^{-2 (\frac{1}{2}+ \mu)}, \] 
which yields \eqref{stopor}, and with \eqref{estdata} 
we obtain $\|\xkd{k(\delta)} -\xd\| = O(\delta^{\frac{2\mu}{2\mu+1}})$, 
which proves the result in the optimal case. 

In case that  $\mu +\frac{1}{2} > \frac{\beta+1}{4}$, the corresponding estimate 
is 
\[ (\tau-1) \delta \leq C k(\delta)^{-(\frac{1}{2}+ \mu + \frac{\beta+1}{4})}, \] 
from which the result in the second case follow. 
\end{proof} 

These rates agree with those of \cite{Ne17} when setting $\beta=1$. 
There, however, only the suboptimal case 2. was possible. 
Our improvement is to show 
that we may achieve optimal order results even with the discrepancy 
principle provided $\beta$ is sufficiently large.

\begin{remark}\label{rem2}
It is clear that in practice $\beta$ should be selected  in 
the regime of optimal rates, i.e. $\beta > 4\mu-1$ for a prior choices and 
$\beta > 4 \mu +1$ for the discrepancy principle. However, it is a rule of 
thumb to choose such parameter also as small as possible, 
or more precisely, in such a way to come close to the 
saturation point, i.e. $\beta \sim 4 \mu -1$, respectively 
$\beta \sim 4 \mu +1$. 
\end{remark}  
 
 \begin{remark}
For  semiiterative methods, a modified discrepancy principle \cite{Ha91,EHN96}
has been defined, where 
the residual in \eqref{discprin} is replaced by an expression of the form 
$(\yd,s_k(A A^*)\yd)$ with a constructed function $s_k$.
This yields an  order-optimal method as for the a priori stopping rule. 
An  adaption of  this strategy 
for Nesterov iteration is certainly possible  and this should yield 
order-optimal rates for all $\mu \leq \frac{\beta +1}{4}$. 
However, the strategy is quite involved and it is not 
completely clear to us how to include this into the iteration 
efficiently. We thus do not intend to investigate such modifications in this 
article. 
\end{remark} 
 
\section{Numerical results} 
In this section we present some small numerical experiments to illustrate 
the semi-saturation phenomenon and to investigate the performance of Nesterov's iteration, 
in particular, with respect to  the optimal-order results.  
 
In a first example we  consider a simple diagonal operator $A = {\rm diag}(\frac{1}{n^2})$, 
for $n =1,\ldots 1000$, as well as an exact solution $\xd = (\frac{1}{n^4} (-1)^n)_{n=1}^{1000}$, 
 which amounts to a source condition being satisfied with 
index $\mu = 0.75$.  Thus, we are in a case of higher smoothness, where the results of 
the present article really improve those of \cite{Ne17}. We add standard normally distributed 
Gaussian noise to the exact data and performed various iterative regularization schemes:
Landweber iteration, the $\nu$-method, and the Nesterov iteration, the latter two with 
various settings of the parameters $\nu$ and $\beta$, respectively.  

We calculated the stopping index either by the discrepancy principle \eqref{discprin} with 
$\tau = 1.01$ or, since we have the luxury of an available exact solution in this 
synthetic example, we also calculate the oracle stopping index, which is defined as 
\[ k_{opt} = {\rm argmin}_k \| \xkd{k} - \xd\|. \]
In other words, $k_{opt}$ is the theoretically optimal possible stopping index. 

\begin{figure} 
\begin{center}
%\begin{minipage}{0.49\textwidth} 
\includegraphics[width=0.45\textwidth]{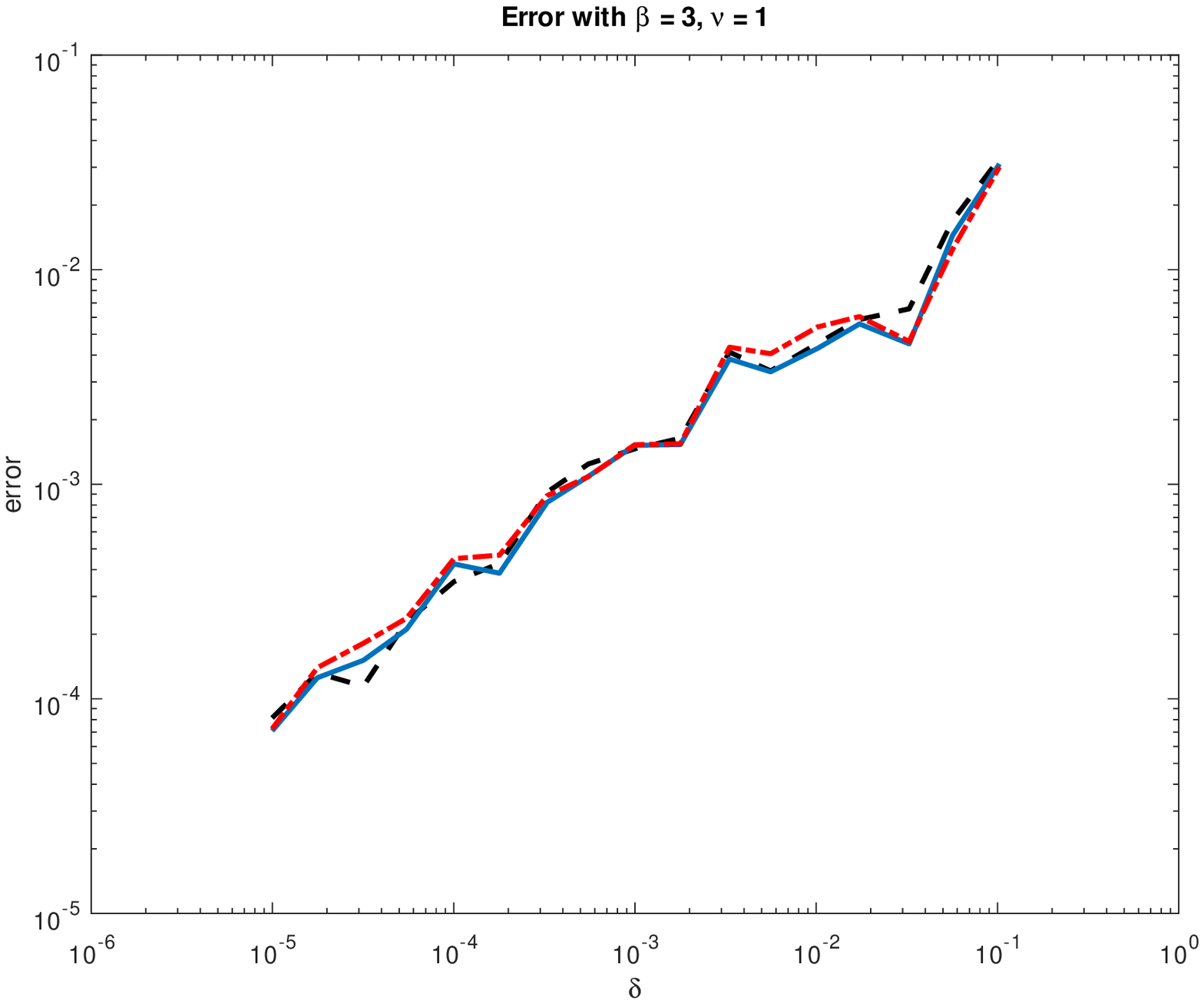} 
%\end{minipage} 
%\begin{minipage}{0.49\textwidth} 
\includegraphics[width=0.45\textwidth]{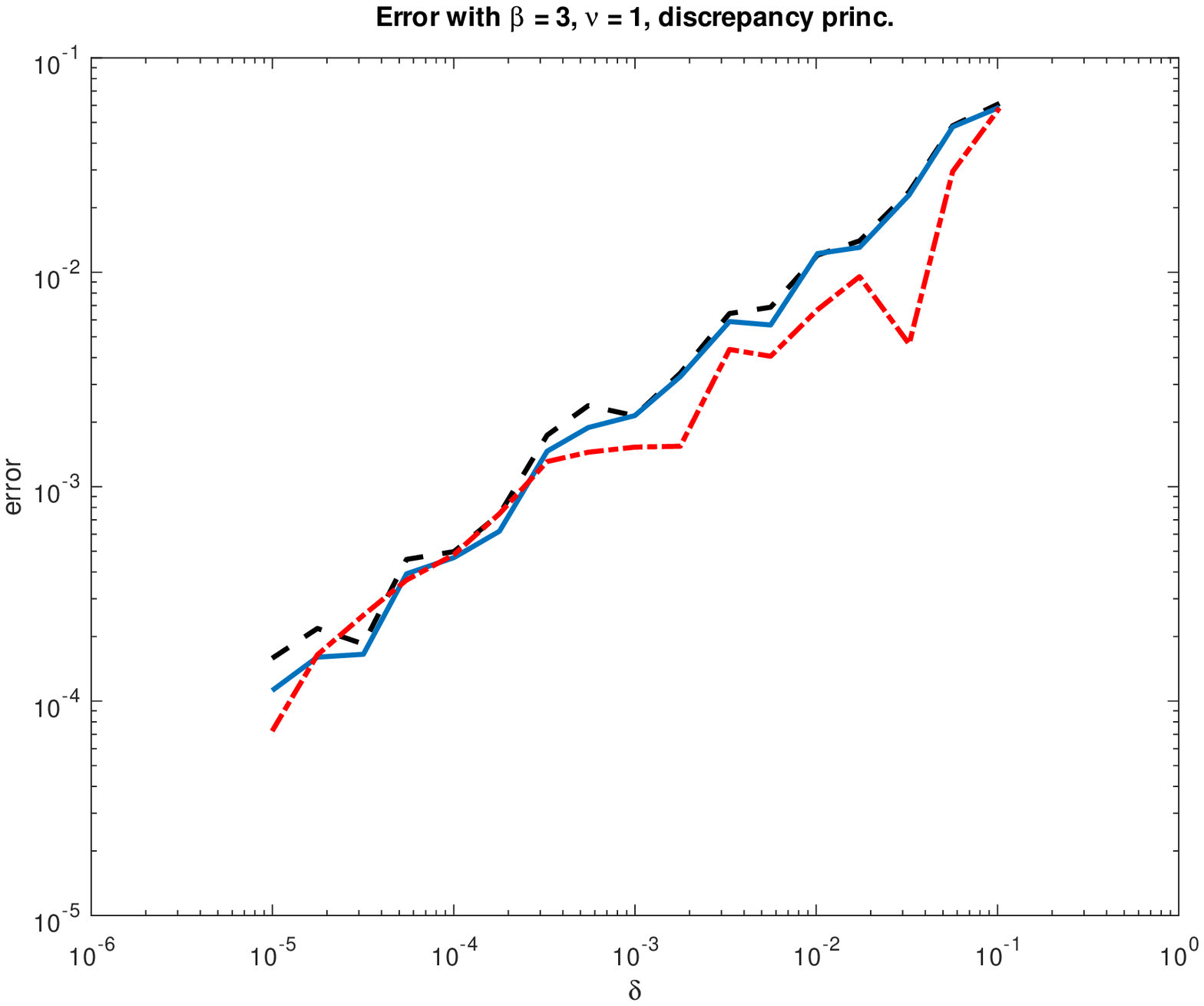}
\caption{Log-log plot of the error $\|\xkd{k(\delta)} -\xd\|$ versus the noiselevel $\delta$
for Nesterov iteration (full line, blue), Landweber iteration (dotted line, black), and the 
$\nu$-method (dashed dotted line, red). Left: optimal stopping rule. Right: stopping by 
discrepancy principle. The parameters $\beta,\nu$ are in an optimal-order regime.} \label{fig1}
%\end{minipage} 
\end{center}
\end{figure} 
\begin{figure} 
\begin{center}
\includegraphics[width=0.45\textwidth]{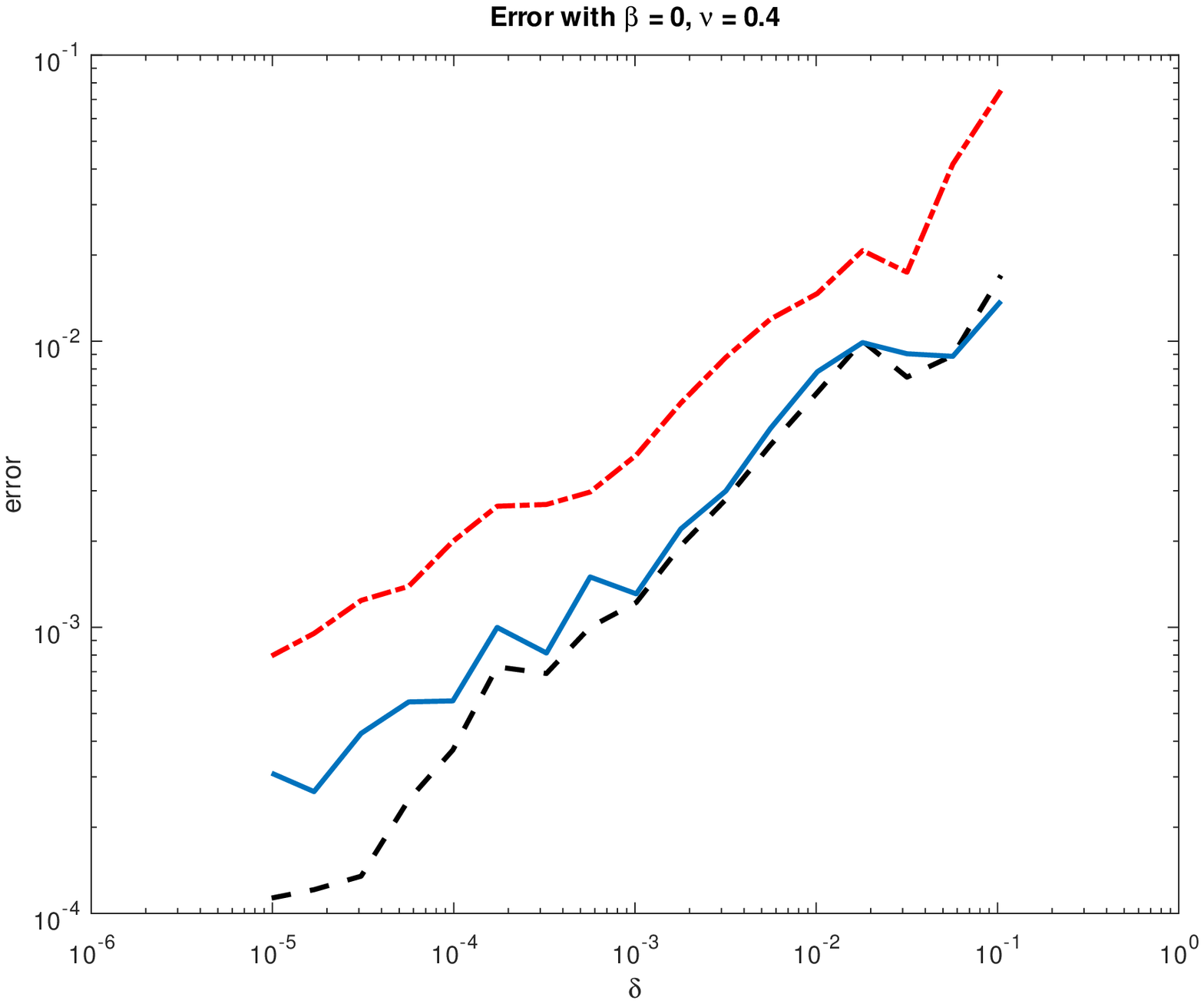}
\includegraphics[width=0.45\textwidth]{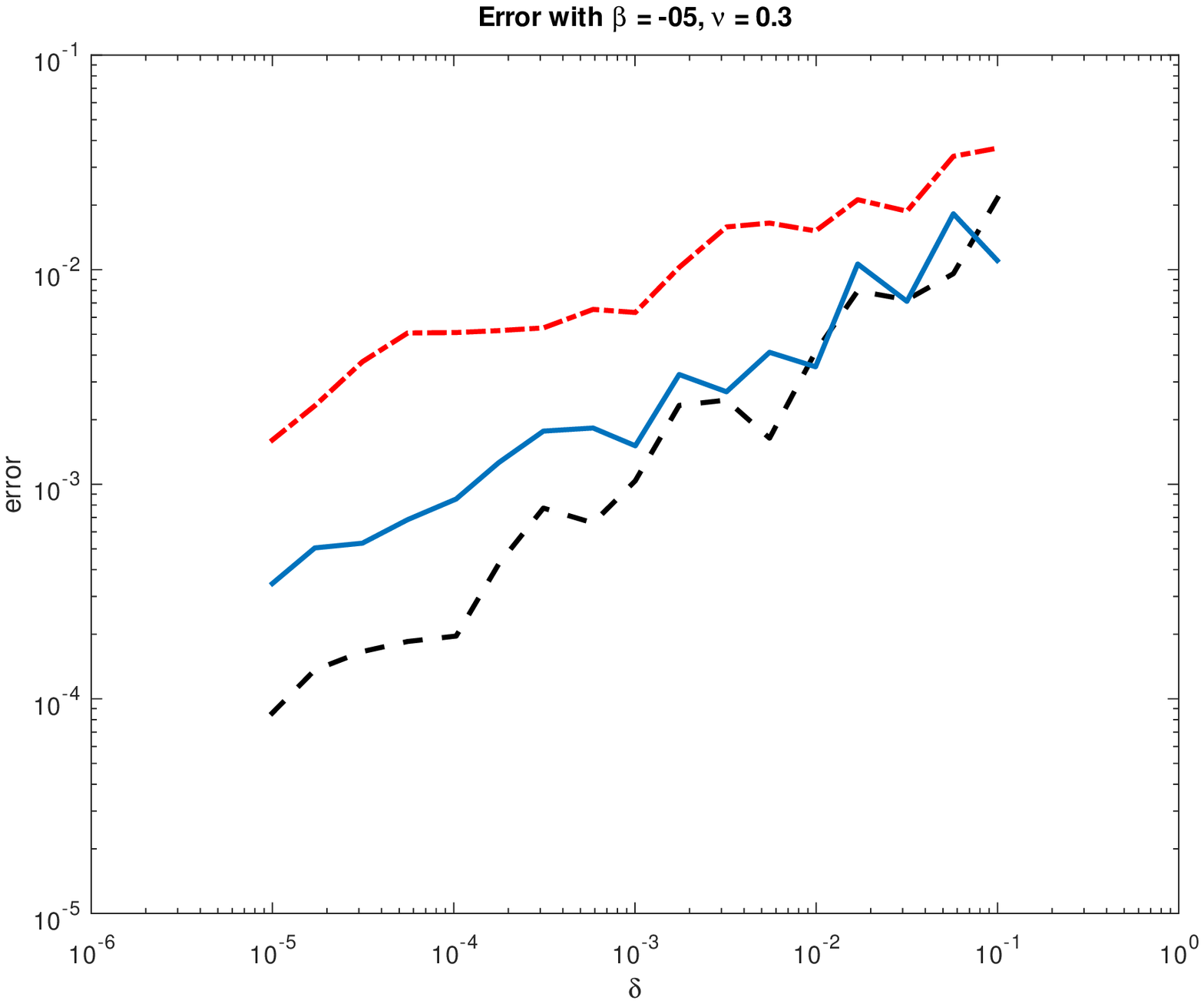}
\caption{The similar plot as in Figure~\ref{fig1}(left) for various iteration parameter in 
a suboptimal-order regime. 
Left: $\nu= 0.4$ and $\beta = 0$. Right: $\nu = 0.3$ and $\beta = -0.5$. Stopping by 
optimal stopping rule $k_{opt}$. }\label{fig2}
\end{center}
\end{figure} 

In Figure~\ref{fig1}, we display the error $\|\xkd{k(\delta)} -\xd\|$ against various noiselevels 
on a log-log scale. The curves correspond to convergence rates for 
Nesterov iteration (full line, blue), Landweber iteration (dotted line, black), and the 
$\nu$-method (dashed dotted line, red). The parameter were chosen as $\beta = 4$ and $\nu = 1$, 
i.e., we are in the optimal-order case covered by item 1 in Theorem~\ref{th:six} and
Theorem~\ref{th:seven}. On the left-hand side we employ the oracle stopping rule using $k_{opt}$
and on the right-hand side we use the discrepancy principle. 

As can be observed, all three methods show a similar (optimal-order) rate, as stated in 
Theorems~\ref{th:six} and \ref{th:seven}. In particular, this verifies one of our findings that the discrepancy 
principle for Nesterov's iteration leads to an optimal-order method provided $\beta$ is chosen appropriately. 

In Figure~\ref{fig2}, we illustrate the semi-saturation phenomenon: Here $\beta$ and $\nu$ are 
deliberately chosen as too small  ($\beta = 0$, $\nu= 0.4$ on the left-hand side and 
$\beta = -0.5$, $\nu = 0.3$  on the right-hand side). We observe that for small $\nu$, 
the convergence rate of the $\nu$-method is slow as a result of its saturation. On the other hand, 
the Nesterov iteration also has a slower rate than the non-saturating Landweber iteration, but, 
as can be expected from our residual polynomial representation, 
it is in between the other two. 
 
We remark that the $\nu$-methods show some unpleasant behaviour when $\nu$ is chosen 
small. The residual is highly oscillating and for small noiselevel we could not even 
reach the prescribed discrepancy, and 
if we did, then the number of iteration was quite high, 
even higher than for Landweber iteration. This might be attributed to our quite aggressive setting 
of the discrepancy principle with $\tau = 1.01$.  
In that respect, the Nesterov iteration was 
very well-behaved, and we had no problem with a small  $\beta$, which is probably due to 
the robust  Landweber-component in the representation \eqref{resibesi}.

The optimal-order convergence only partly illustrates the effective performance of the 
methods. In Tables~\ref{tab1} we therefore provide   the 
ratio of  errors values, i.e., the numbers in the table are 
$\frac{\|\xkd{method,k} - \xd\|}{\|\xkd{Nesterov,k_{opt}} - \xd\|}$, where 
$\xkd{Nesterov,k_{opt}}$ denotes Nesterov iteration with the optimal stopping rule and 
$\xkd{method,k}$ the iteration of the respective method with the respective stopping rule. 
All results correspond to an optimal-order regime of parameters
(those of Figure~\ref{fig1}).
The number of iterations 
(both for the oracle stopping rule and the discrepancy principle) are given in 
Table~\ref{tab2}. 
In these tables,  we also include the corresponding results 
for the conjugate gradient iteration CGNE \cite{Ha95}. 

\begin{table}
 \caption{Errors compared to Nesterov iteration: $\frac{\|\xkd{method,k} - \xd\|}{\|\xkd{Nesterov,k_{opt}} - \xd\|}$. }\label{tab1}
\begin{tabular}{cc|ccccc}
& & \multicolumn{5}{c}{$\delta$} \\ 
Method & Stopping  & $10^{-5}$ & $10^{-4}$ & $10^{-3}$ & $10^{-2}$ & $10^{-1}$ \\ \hline
Nesterov  & $k_{opt}$ & 1 & 1 & 1 & 1 & 1   \\
Landweber  & $k_{opt}$ &  1.15 &   0.83 &  0.96 &  1.05  & 1.06 \\
$\nu$-Method  & $k_{opt}$ &  1.02 &  1.06 &  1.01 &  1.26&  0.97 \\
CGNE & $k_{opt}$ &  1.02 &  0.82 &  1.05 &  1.02  &  0.84   \\
Nesterov  & Discrepancy &  1.58  & 1.10 &  1.41 &  2.84  & 1.90 \\
Landweber  & Discrepancy  &   2.23 & 1.17 &  1.41  & 2.80  & 1.98 \\
$\nu$-Method  & Discrepancy &  1.02 &  1.13 &  1.00 &  1.56  & 1.88  \\
CGNE  & Discrepancy &  1.81  & 1.19 &  1.05 &  2.51  & 1.97   \\
 \end{tabular}
\end{table}

\begin{table}
  \caption{Number of iterations for various methods; setting as in Table~\ref{tab1}.}\label{tab2}
\begin{tabular}{cc|ccccc}
& & \multicolumn{5}{c}{$\delta$}  \\ 
Method & Stopping  & $10^{-5}$ & $10^{-4}$ & $10^{-3}$ & $10^{-2}$ & $10^{-1}$ \\ \hline
Nesterov  & $k_{opt}$ & 371  & 163  &  65  &  26  &  15 \\
Landweber  & $k_{opt}$ &  11000  &  2193  &   512  &   145   &   36  \\
$\nu$-Method  & $k_{opt}$ &   190  &  82  &  33   & 22   &  9 \\
CGNE & $k_{opt}$ &  10  &  6  &  4  &  3  &  2   \\
Nesterov  & Discrepancy &   260  & 111 &   39  &  13   &  1 \\ 
Landweber  & Discrepancy &  5106  & 1080 &   220  &   37    &  1    \\
$\nu$-Method  & Discrepancy &  190  &  96  &  33  &  10   &  1   \\
CGNE  & Discrepancy & 8 &  5 &  4 &  2  & 1  \\
 \end{tabular}
\end{table}

In terms of the number of iteration, the Nesterov iteration is slightly slower 
than the $\nu$-methods (approximately by a constant factor of 1.5) but both have a similar modest increase 
of iterations when $\delta$ is 
decreased. Both need more iteration 
than the CGNE-method, which, of course, is the fastest one by design.
The slightly higher number of iterations might be attributed to the {\em better}
error estimate in \eqref{estdata}. (Note that the $\nu$-methods have a $2$ in place 
of $\sqrt{2}$ there). It might appear a little bit paradoxical that a better estimate 
leads to slower convergence, but this is clear from the theory as the number of 
iteration is a decreasing function of  $\delta$ and thus also of any  factor in front of 
$\delta$.  This factor, however, pays off when considering the total error of the method, 
and we observe that  Nesterov iteration with the optimal choice $k_{opt}$ 
indeed has almost always a slightly smaller error than 
the $\nu$-method. 
Surprisingly, it is in several instances also better than the CGNE-method.
However, the Nesterov method sometimes loses some of its advantages 
against the $\nu$-method, when using the discrepancy principle, but the performance
is still acceptable. 

Some further experiments indicate that the results are rather insensitive to 
overestimating $\beta$. As stated in Remark~\ref{rem2}, the best choice is usually
related to the smoothness index, but there was arose no serious problems when 
$\beta$ was larger.

Further numerical experiments have been performed in \cite{Ne17}: Even though the 
value of $\beta$ was not reported there, the results are consistent with our theory 
with the choice $\beta = 1$. The forward operator there was the Green's function for 
the solution of the 1D boundary value problem $-u'' = f$ with homogeneous boundary conditions. 
Exact solutions with various smoothness  are stated there:  Example~5.1 with $\mu=\frac{1}{8}$, 
Example~5.2 with $\mu = \frac{5}{8}$, and Example~5.3 with $\mu = \frac{17}{8}$. 
We used the same problem and the same examples, but we calculated $A$ by using a FEM-discretization
of the boundary value problem and $A$ as the corresponding solution operator. 
For simplicity we ignored discretization errors and took the discretized (projected) 
solution as $\xd$. 
\begin{figure}[t]
\begin{center}
\includegraphics[width=0.45\textwidth]{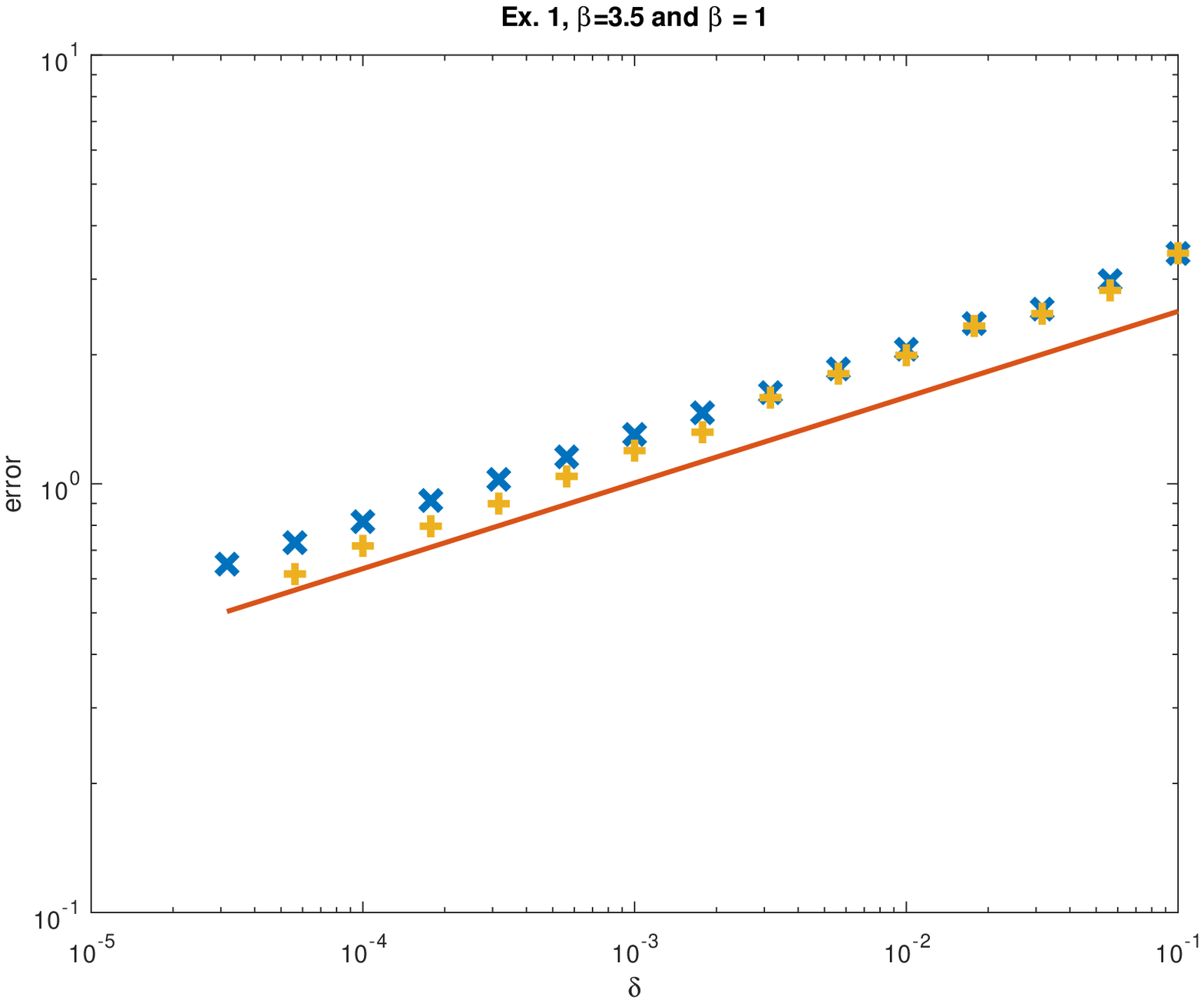} 
\includegraphics[width=0.45\textwidth]{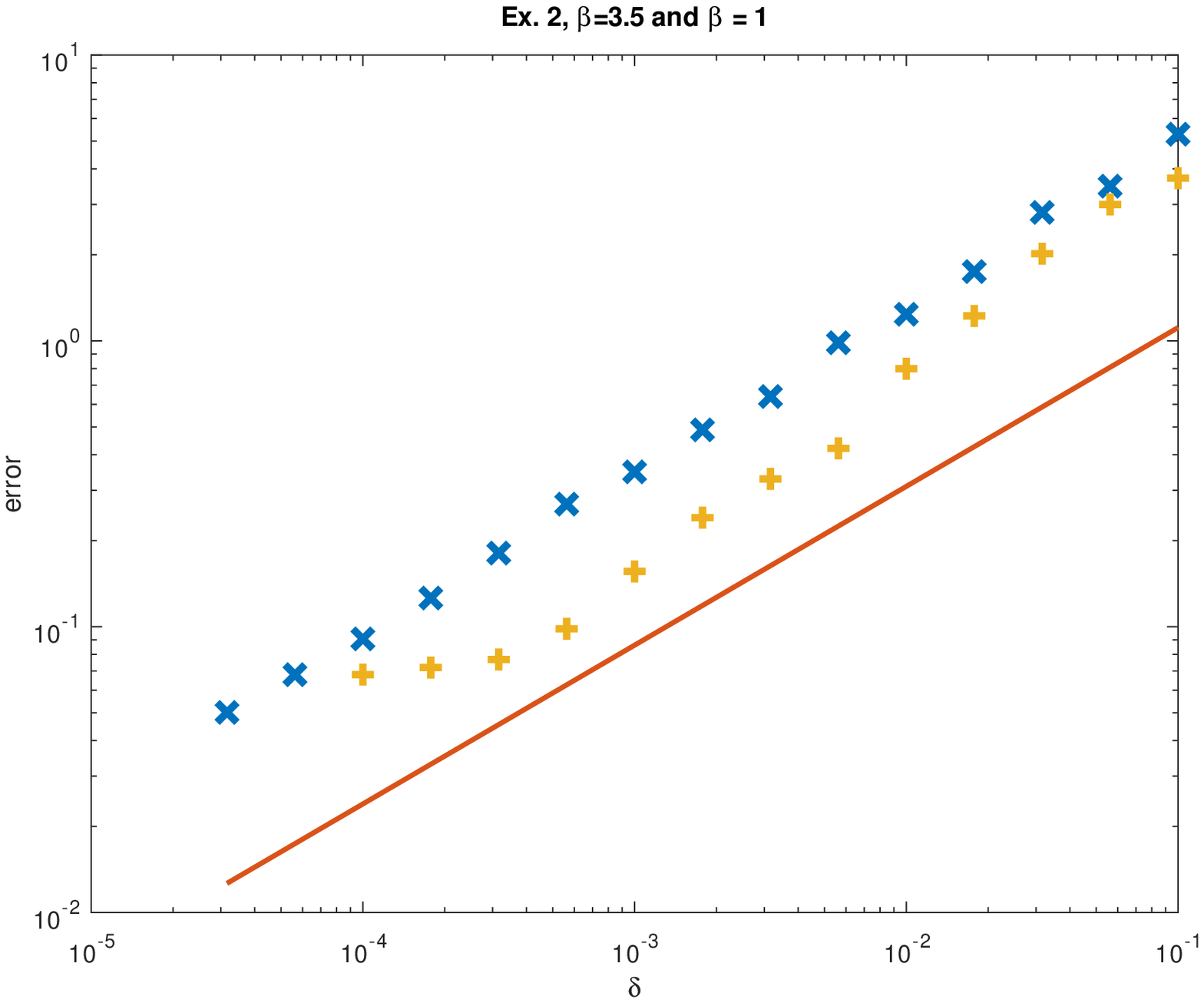} 
\includegraphics[width=0.45\textwidth]{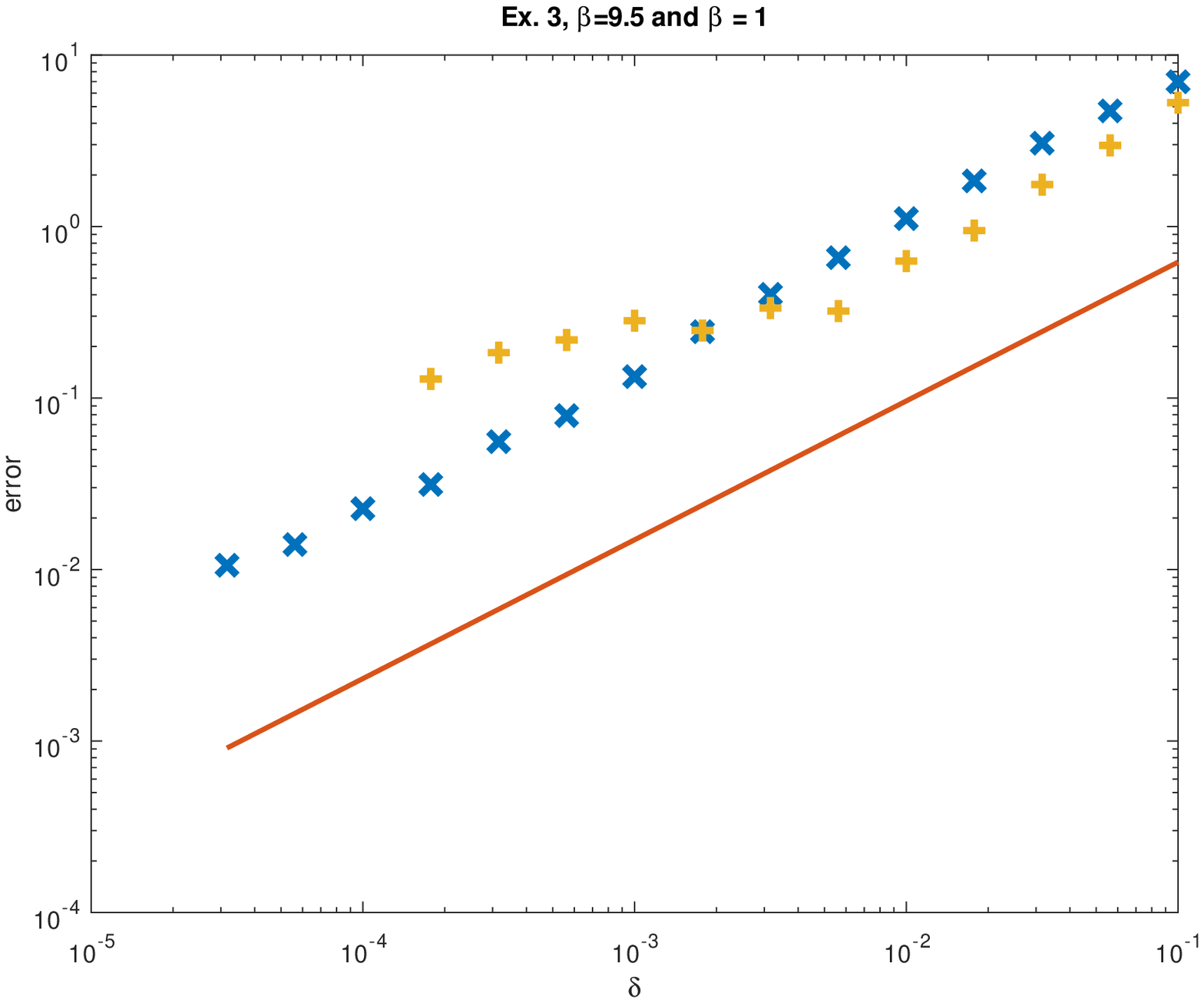} 
\caption{Convergence rates for the examples in \cite{Ne17}. Left: Ex.~1, smoothness index 
$\mu=\frac{1}{8}$. Center: Ex.~2, smoothness index 
$\mu=\frac{5}{8}$. Right:  Ex.~3, smoothness index 
$\mu=\frac{17}{8}$. Displayed are the errors versus the noiselevel on a logarithmic scale. 
A marker 'x' indicates optimal choice of $\beta$, and '+' indicates suboptimal choice $\beta = 1$. 
The full line indicates the optimal order rate. }\label{todo}
\end{center} 
\end{figure} 

The main purpose of this experiment is to verify that the discrepancy principle
($\tau = 1.1$) can be made  
an optimal-order method. We choose $\beta = 3.5$  for the first two examples and 
$\beta = 9.5$ for the third, which should in any case lead to an optimal-order situation.  
In Figure~\ref{todo}, we plotted the error versus the relative noiselevel on a logarithmic
scale for the three examples with this choice of $\beta$, indicated by the marker 'x'. 
As a comparison, we also 
indicated the predicted optimal rate by a solid line. Furthermore, also shown and marked with 
'+' are 
the corresponding results for $\beta = 1$, i.e., in 
 the suboptimal case. 

These results clearly  illustrate that for the discrepancy principle we may achieve the optimal 
order rates with the correct choice of $\beta$ and for a wrong choice of $\beta$ the rate 
deteriorates. For low-smoothness as in Example~1 (left picture in Figure~\ref{todo}, however, 
there seems to occur almost no deterioration contrary to expectation.

\section{Conclusion} 
We have provided a representation of the residual polynomials for Nesterov's acceleration method 
for linear ill-posed problems as a product  of Gegenbauer polynomials and Landweber-type residuals. 
This allowed us to prove optimal-order rates for an a priori stopping rule and the discrepancy 
principle as long as $\beta$ in \eqref{achoice} is sufficiently large. The number of iteration 
is shown to be of the same order as for other fast methods such as the $\nu$-method or the 
conjugate gradients methods. Moreover, our representation clearly explains the observed 
semi-saturation phenomenon. 

Within the class of linear iterative methods, the Nesterov acceleration is an excellent choice, 
as it is a fast method as well as a quite robust one. Although, it must be conceded,
that it cannot compete with the conjugate gradient method in terms of number of iterations. 
However, this is compensated by its flexibility and simplicity of use, which also allows one
to easily integrate it into existing gradient methods and also   to apply it 
in nonlinear cases.

\bibliographystyle{siam}
\bibliography{nest}

\end{document}